\documentclass{amsart}

\usepackage{amscd}
\usepackage{amssymb}

\usepackage{color}

\usepackage{amsthm}

\usepackage{hyperref}
\hypersetup{
  colorlinks=true, 
   linkcolor=red,  
}

\newtheorem{thm}[equation]{Theorem}
\newtheorem{lem}[equation]{Lemma}
\newtheorem{cor}[equation]{Corollary}
\newtheorem{prop}[equation]{Proposition}

\newtheorem*{thm*}{Theorem}
\newtheorem*{prop*}{Proposition}
\newtheorem*{cor*}{Corollary}
\newtheorem*{lem*}{Lemma}
\newtheorem*{MT*}{Main Theorem}

\newtheorem*{ques*}{Question}

\theoremstyle{definition} %

\newtheorem*{defn*}{Definition}

\newtheorem{eg}[equation]{Example}

\theoremstyle{remark} %
\newtheorem{rmk}[equation]{Remark}

\newtheorem*{rmk*}{Remark}
\newtheorem*{rmks*}{Remarks}

\DeclareMathOperator{\Spin}{Spin}           
\DeclareMathOperator{\Sp}{Sp}
\DeclareMathOperator{\PSp}{PSp}
\DeclareMathOperator{\SL}{SL}
\DeclareMathOperator{\GL}{GL}
\DeclareMathOperator{\PGL}{PGL}
\DeclareMathOperator{\PSO}{PSO}
\DeclareMathOperator{\HSpin}{HSpin}

\newcommand{\SO}{\mathrm{SO}}

\DeclareMathOperator{\Lie}{Lie}
\DeclareMathOperator{\Ad}{Ad}
\newcommand{\ot}{\otimes}

\newcommand{\Z}{\mathbb{Z}}
\newcommand{\F}{\mathbb{F}}
\newcommand{\C}{\mathbb{C}}

\newcommand{\Vb}{\overline{V}}

\newcommand{\e}{\varepsilon}

\newcommand{\pgl}{\mathfrak{pgl}}
\newcommand{\gl}{\mathfrak{gl}}

\newcommand{\stbtmat}[4]{\left( \begin{smallmatrix} #1&#2 \\ #3&#4 \end{smallmatrix} \right) }

\newcommand{\qform}[1]{{\left\langle{#1}\right\rangle}}                   

\DeclareMathOperator{\car}{char}
\DeclareMathOperator{\rk}{rank}

\newcommand{\Gm}{\mathbb{G}_m}
\newcommand{\Ga}{\mathbb{G}_a}

\renewcommand{\P}{\mathbb{P}}
\renewcommand{\sc}{{\mathrm{sc}}}
\newcommand{\adj}{{\mathrm{adj}}}

\DeclareMathOperator{\trdeg}{trdeg}
\DeclareMathOperator{\res}{res}

\DeclareMathOperator{\AGL}{AGL}

\newcommand{\Aff}{\mathbb{A}}

\DeclareMathOperator{\diag}{diag}

\DeclareMathOperator{\ed}{ed}
\numberwithin{equation}{section}

\begin{document}

\title[Essential dimension of algebraic groups]{Essential dimension of algebraic groups, including bad characteristic}
\author[S. Garibaldi]{Skip Garibaldi}
\address{Garibaldi: Center for Communications Research, San Diego, California 92121}
\email{skip@member.ams.org}

\author[R.M. Guralnick]{Robert M. Guralnick}
\address{Guralnick: Department of Mathematics, University of Southern California,
Los Angeles, CA 90089-2532}
\email{guralnic@usc.edu}

\begin{abstract}
We give upper bounds on the essential dimension of (quasi-)simple algebraic groups over an algebraically closed field that hold in all characteristics.  The results depend on showing that certain representations are generically free.  In particular, aside from the cases of spin and half-spin groups, we prove that the essential dimension of a simple algebraic group $G$ of rank at least two is at most $\dim G - 2(\rk G) - 1$.  It is known that the essential dimension of spin and half-spin groups grows exponentially in the rank.  In most cases, our bounds are as good or better than those known in characteristic zero and the proofs are shorter.  We also compute the generic stabilizer of an adjoint group on its Lie algebra.
\end{abstract}

\thanks{Guralnick was partially supported by NSF grants DMS-1265297 and DMS-1302886.}

\maketitle

\setcounter{tocdepth}{1}
\tableofcontents

\section{Introduction}

The essential dimension of an algebraic group $G$ is the minimal transcendence degree of the field of definition of a versal $G$-torsor.  (Although inaccurate, one can think of it as the number of parameters needed to specify a $G$-torsor.)  This invariant captures deep information about algebraic structures with automorphism group $G$, and it is difficult to calculate.  For example, the fact that $\ed(\PGL_2) = \ed(\PGL_3) = 2$ corresponds to the classical fact that division algebras of dimension $2^2$ or $3^2$ over their center are cyclic, and it is an open problem whether the essential dimension of $\PGL_p$ is 2 for primes $p \ge 5$ \cite[Problem 6.2]{ketura}, although it is known that $\ed(\PGL_n)$ is not $O(n)$ \cite{M:lower}.  Therefore, the bulk of known results on essential dimension provide upper or lower bounds, as in \cite{ChSe}, 
 \cite{GiRei},
 \cite{BaekM}, \cite{MR3039772},  etc.  (See \cite{M:bbk}, \cite{M:ed}, or \cite{Rei:ICM} for a survey of the current state of the art.)  
In this paper, we provide upper bounds on $\ed(G)$ for every simple algebraic group $G$ over an algebraically closed field $k$, regardless of the characteristic of $k$.  Our bounds are in some cases as good as (Theorem \ref{A}) or better (Theorems \ref{big.O} and \ref{PSp}) than the bounds known in characteristic zero, and have shorter proofs.  One summary consequence of our results is the following.

\begin{thm}  \label{big.O}
Let $G$ be a simple algebraic group over an algebraically closed field. 
Then 
\[
\ed(G) \le  \dim G - 2 (\rk G) - 1
\]
or $G \cong \PGL_2$, or $G \cong \Spin_n$ or $\HSpin_n$ for some $n$.
\end{thm}

For the excluded cases, $\ed(\PGL_2) = 2$ (Example \ref{A1.ed}).  For spin and half-spin groups, essential dimension grows exponentially in $n$ \cite{BRV} whereas the dimension, $\binom{n}{2}$, is quadratic in $n$.   Specifically,
$\ed(\Spin_n) > \dim \Spin_n$ for all $n \ge 19$ and $\ed(\HSpin_n) > \dim \HSpin_n$ for $n$ divisible by 4 and $\ge 20$.  

\subsection*{Adjoint groups}
Under the additional hypotheses that $G$ is adjoint and $\car k = 0$,  it is well known that an adjoint semisimple group $G$ acts generically freely on $\Lie(G) \oplus \Lie(G)$\footnote{By, e.g., \cite[Lemma 3.3(b)]{Richardson:n}.  For analogous statements in prime characteristic, see \S\ref{adj.sec}.} and consequently $\ed(G) \le \dim G$, as was pointed out in \cite[Remark 3-11]{BRV}.  In this setting, the stronger bound in Theorem \ref{big.O} was proved in \cite{Lemire}.  Dropping the hypothesis on $\car k$ but still assuming $G$ is adjoint, the bound $\ed(G) \le \dim G - 2(\rk G)$ was recently proved in \cite[Cor.~10]{BGS}.

Theorem \ref{big.O} for adjoint groups includes the following bounds, where we write $T^\adj_n$ for an adjoint group of type $T_n$:
\begin{gather} \label{BGS.1}
\ed(E^\adj_6) \le 65, \quad \ed(E^\adj_7) \le 118, \quad \ed(E_8) \le 231, \\
\text{and} \quad \ed(D_n^\adj) \le \text{$2n^2 - 3n - 1$ for $n \ge 4$.} \notag
\end{gather}
(The adjoint group $D_n^\adj$ is sometimes denoted $\PSO_{2n}$.)  These bounds agree with those in \cite{Lemire} for characteristic 0.  (The number 112 given there for $E_7^\adj$ was a typo.)  

We remark that the essential dimension of $\SO_{2n+1}$ (adjoint of type $B_n$) is $2n$ if $\car k \ne 2$ \cite{Rei:ed} and is $n+1$ if $\car k =2$ \cite{BabicCh}.  For $\SO_{2n}$ (of type $D_n$), the essential dimension is $2n-1$ if $\car k \ne 2$ and is $n$ or $n+1$ if $\car k = 2$. 

\subsection*{Groups of type $C$}
We give the following upper bound for adjoint groups of type $C_n$, which improves on the bound $2n^2 - 3n -1$ given in \cite{Lemire} in characteristic zero.

\begin{thm} \label{PSp}
Over an algebraically closed field $k$ and for $n \ge 4$:
\[
\ed(\GL_{2n}/\mu_2) \le \ed(\PSp_{2n}) \le \begin{cases}
2n^2 - 3n - 4 & \text{if $\car k \!\not\vert\,  n$ or $n=4$} \\
2n^2 - 3n - 6 & \text{if $\car k \mid n$ and $n > 4$.}
\end{cases}
\]
\end{thm}

The interesting case of Theorem \ref{PSp} is when $n$ is even; in the special case where $n$ is odd, the natural map $\PGL_2 \times \SO_n \hookrightarrow \PSp_{2n}$ gives a surjection $H^1(k, \PGL_2) \times H^1(k, \SO_n) \to H^1(k, \PSp_{2n})$, and $\ed(\PSp_{2n}) = n + 1$ for $n$ odd, cf.~\cite[p.~302]{MacD:CIodd}.

It is known that $\ed(\GL_8/\mu_2) = 8$ if $\car k \ne 2$ \cite[Cor.~1.4]{BaekM} and is $\le 10$ if $\car k = 2$ \cite[Cor.~1.4]{Baek:edp}, somewhat better than the bound $\le 16$ provided by Theorem \ref{PSp}. 

\subsection*{Groups of type $A$}
The essential $p$-dimension of $\GL_n / \mu_m$ and of $\SL_n / \mu_m$ have been studied in \cite{BaekM}, \cite{Baek:edp}, and \cite{ChM:edpA}.  Here and in \eqref{GLSL.bound} we give upper bounds for the essential dimension (without the $p$).

\begin{thm} \label{A}
Over an algebraically closed field and for $m$ dividing $n \ge 4$, we have:
\[
\ed(\PGL_n) \le n^2 - 3n + 1 \quad \text{and} \quad 
\ed(\SL_n / \mu_m) \le n^2 - 3n + n/m + 1.
\]
\end{thm}
 
 If $m = 1$, then $\ed(\SL_n) = 0$.  If $m = n$, then $\SL_n / \mu_m = \PGL_n$.  Our bound for $\PGL_n$ agrees with the one given by \cite{Lemire} in $\car k = 0$; we remove this hypothesis.  A better bound on $\ed(\PGL_n)$ is known for $n$ odd \cite{LRRS}.  See \cite[\S10]{M:ed} or \cite[\S7.6]{Rei:ICM} for discussions of the many more results on upper bounds for $\PGL_n$.  If $m = 2$, then (applying Lemma \ref{GLSL}) the bound in Theorem \ref{PSp} is better by about a factor of 2.

\subsection*{Exceptional groups}
Concerning exceptional groups, a series of papers \cite{Lemire}, \cite{MacD:F4}, \cite{MacD:E7}, \cite{LoetscherMacD} have led to the following upper bounds for exceptional groups:
\[
\ed(F_4) \le 7, \quad \ed(E_6^\sc) \le 8, \quad \text{and} \quad \ed(E_7^\sc) \le 11 \quad \text{ if $\car k \ne 2, 3$.}
\]
(Here $F_4$, $E_6^\sc$ and $E_7^\sc$ denote simple and simply connected groups of types $F_4$, $E_6$, and $E_7$; the displayed upper bounds are meant to be  
compared with the dimensions of 52, 78, and 133 respectively.  These upper bounds are close to the known lower bounds of 5, 4, and 8 for $\car k \ne 2, 3$.)
The proofs of these upper bounds for $F_4$ and $E_7^\sc$ are technical and detailed calculations.  The following weaker bounds have the advantage of simple proofs and holding for fields of characteristic 2 and 3.

\begin{thm} \label{ed.thm}
Over an algebraically closed field, we have:
\[
\ed(F_4) \le 19,
\qquad
\ed(E_6^\sc) \le 20, \qquad \text{and}
\qquad
\ed(E_7^\sc) \le 49.
\]
\end{thm}

The proofs of most of the theorems above rely on computations of the (scheme-theoretic) stabilizer of a generic element in a representation of $N_G(T)$ for $T$ a maximal torus in $G$.   The proof of Theorem \ref{PSp} uses  the computation of a generic stabilizer in the
 action of $\Sp$ on $L(\lambda_2)$.  Using the same technique, we prove calculate the generic stabilizer of an adjoint group $G$ acting on its Lie algebra. In particular, this stabilizer is connected unless $\car k = 2$.  In the final section, we give upper bounds on essential dimension for small spin and half-spin groups, completing the list of upper bounds on $\ed(G)$ for $G$ simple and connected over an algebraically closed field.
 
\section{Generically free actions}
Let $G$ be an affine group scheme of finite type over a field $k$, which we assume is algebraically closed.  (If $G$ is additionally smooth, then we say that $G$ is an \emph{algebraic group}.)  We put $G^\circ$ for the identity component of $G$.  If $G$ acts on a variety $X$, the stabilizer $G_x$ of an element $x \in X(k)$ is a sub-group-scheme of $G$ with points
\[
G_x(R) = \{ g \in G(R) \mid gx = x \}
\]
for every $k$-algebra $R$.  Statements ``for generic $x$'' means that there is a dense open subset $U$ of $X$ such that the property holds for all $x \in U$. 

Suppose $G$ acts on a variety $X$ in the sense that there is a map of $k$-schemes $G \times X \to X$ satisfying the axioms of a group action.  We say that $G$ acts \emph{generically freely} on $X$  if there is a nonempty open subset $U$ of $X$ such that for every $u \in U$ the stabilizer $G_u$ is the trivial group scheme 1.  It is equivalent to require that $G_u(k) = 1$ and $\Lie(G_u) = 0$.  Indeed, if $\Lie(G_u) = 0$, then $G_u$ is finite \'etale and (since $k$ is algebraically closed) it follows that $G_u = 1$.  

\begin{eg} \label{torus}
For $T$ a diagonalizable group scheme (e.g., a split torus) 
acting linearly on a vector space $V$, the stabilizer $T_v$ of a generic vector $v \in V$ is $\cap_{\omega \in \Omega} \ker \omega$ where $\Omega \subset T^*$ is the set of weights of $V$, i.e., $T_v$ is the kernel of the action.  (By the duality between diagonalizable group schemes and finitely generated abelian groups, this is a statement on the level of group schemes.) In particular, $T$ acts generically freely on $V$ if and only if $\Omega$ spans $T^*$.

Similarly, the stabilizer $T_{[v]}$ of a generic element $[v] \in \P(V)$ is $\cap_{\omega, \omega'} \ker(\omega- \omega')$, so $T$ acts generically freely on $\P(V)$ iff the set of differences $\omega - \omega'$ span $T^*$.
\end{eg}

For other groups $G$, we have the following well known lemma, see for example \cite[Lemma 2.2]{GG:simple}.

\begin{lem} \label{point}
Suppose $G$ is connected and $X$ is irreducible.  If there is a field $K \supseteq k$ and an element $x_0 \in X(K)$ such that $G_{x_0}$ is finite \'etale, then there is an $n \ge 1$ and a nonempty open $U \subseteq X$ such that, for every algebraically closed field $E \supseteq k$ and every $u \in U(E)$, $G_u$ is finite \'etale and $|G_u(E)| = n$. $\hfill\qed$
\end{lem}

Note that finding some $x_0$ with $G_{x_0} = 1$ does not imply that $G$ acts generically freely on $X$; it is common that such an $x_0$ will exist in cases where $G_x$ is finite \'etale but $\ne 1$ for generic $x$.  This was pointed out already in \cite{AndreevPopov}; see \cite{GurLawther} for more discussion and examples.

Nonetheless, Lemma \ref{point} may be used to prove that an action is generically free as follows.  Suppose $G$, $X$ and the action of $G$ on $X$ can be defined over a countable algebraically closed field $F$ and that $X$ is unirational, i.e., there is an $F$-defined dominant rational map $\phi \!: \Aff^d \dashrightarrow X$ for some $d$.  Adjoin $d$ indeterminates $a_1, \ldots, a_d$ to $F$ and calculate $G_{x_0}$ for $x_0 = \phi(a_1, \ldots, a_d)$.  As $F$ is countable, for $K$ an uncountable algebraically closed field containing $F$, the elements of $\Aff^d(K)$ with algebraically independent coordinates are the complement of countably many closed subsets so are dense.  Therefore, modifying $\phi$ by an $F$-automorphism of $\Aff^d$, the calculation of $G_{x_0}$ implicitly also calculates $G_x$ for $x$ in a dense subset.  In particular, if $G_{x_0} = 1$, then the lemma gives that $G$ acts generically freely on $X$.

\subsection*{Groups whose identity component is a torus}
Suppose that $G$ is an algebraic group whose identity component is a torus $T$.  As $k$ is assumed algebraically closed, the component group $G/T$ is a finite constant group.  We are interested in representations $V$ of $G$ such that $G$ acts generically freely on $V$ or $\P(V)$.  Evidently it is necessary that $T$ acts faithfully on $V$ or $\P(V)$, respectively.

\begin{lem} \label{proj}
Let $G$ be an algebraic group with identity component a torus $T$.  Suppose that $G$ acts linearly on a vector space $V$ such that:
\begin{enumerate}
\item  \label{proj.mult} every weight of $V$ has multiplicity $1$, and
\item  \label{proj.ker} for $\Omega$ the set of weights of $V$, $G/T$ acts faithfully on the kernel of the map $\psi \!:\!~\!\oplus_{\omega \in \Omega} \Z \mapsto T^*$ given by $(n_\omega) \mapsto \sum_\omega n_\omega \omega$.
\end{enumerate}
If $T$ acts faithfully on $V$ (resp., $\P(V)$), then $G$ acts generically freely on $V$ (resp., $\P(V)$).
\end{lem}

We give a concrete proof.  Alternatively one could adapt the proof of \cite[Lemma 3.3]{MeyerRei:PGL}.

\begin{proof}
As $G$ is the extension of a finite constant group by a torus, it and the representation $V$ are defined over the algebraic closure of the prime field in $k$.
Put $K$ for the algebraic closure of the field obtained by adjoining independent indeterminates $c_\chi$ to $k$ for each weight $\chi$ of $V$.  Fix elements $v_\chi \in V$ generating the $\chi$ weight space for each $\chi$ and put $v := \sum_\chi c_\chi v_\chi \in V \ot K$.  

We have a homomorphism $\delta \!: G_{[v]} \to \Gm$ given by $g \mapsto gv/v$.  Thinking now of the function $\pi \!: V \to \Aff^1$ defined by $\pi(\sum \alpha_\chi v_\chi) = \prod \alpha_\chi$, we see that $\pi(gv) = \pi(\delta(g)v) = \delta(g)^{\dim V} \pi(v)$.  On the other hand, for $\delta'$ denoting the composition $\delta \!: G \to \GL(V) \xrightarrow{\det} \Gm$, as the image of $G$ in $\GL(V)$ consists of monomial matrices, we find that 
 $\pi(gv) = \pm \delta'(g) \pi(v)$.  That is, $\delta(g)^{\dim V} = \pm \delta'(g)$, so the image of $\delta$ in $\Gm$ is the group scheme $\mu_a$ of $a$-th roots of unity for some $a$.

For sake of contradiction, suppose there exists a $g \in (G_{[v]})(K)$ mapping to a non-identity element $w$ in $(G/T)(K)$.  Pick $n \in G(k)$ with the same image $w$, so $g = nt$ for some $t \in T(K)$.  Now $nv_\chi = m_\chi v_{w\chi}$ for some $m_\chi \in k^\times$, and we have an equation
\[
\delta(g) v = ntv = \sum_{\chi} c_\chi \chi(t) m_\chi v_{w\chi},
\]
hence $ \chi(t)  = \delta(g) c_{w\chi} / (c_\chi m_\chi)$ for all $\chi$.

By hypothesis, there exist $\chi_1, \ldots, \chi_r \in \Omega$ and nonzero $z_1, \ldots, z_r \in \Z$ such that $\sum z_i \chi_i = 0$ in $T^*$, yet the tuple $(z_\chi) \in \oplus_{\chi \in \Omega} \Z$ is not fixed by $w$ where $z_\chi = z_i$ if $\chi = \chi_i$ and $z_\chi = 0$ otherwise.
As $\sum z_i \chi_i = 0$, we have
\begin{equation} \label{proj.1}
1 = \prod_i \left( \frac{\delta(g)c_{w\chi_i} }{c_{\chi_i} m_{\chi_i}} \right)^{z_i},
\end{equation}
an equation in $K$, where $\delta(g)$ and the $m_{\chi_i}$ belong to $k^\times$.
But the indeterminates appearing in the numerator correspond to the tuple $(z_{w \cdot \chi})$
whereas those in the denominator correspond to $(z_\chi)$, so the equality \eqref{proj.1} is impossible. 

That is, 
the image of $G_{[v]}$ in the constant group $G/T$ is trivial, and $G_{[v]}$ is contained in $T$.  
Thus $G_v$ or $G_{[v]}$ is trivial by Example \ref{torus}.  Since we have proved that an element with algebraically independent coordinates has trivial stabilizer, $G$ acts generically freely by the discussion following Lemma \ref{point}.
\end{proof}

\begin{eg} \label{proj.m1}
Suppose there is an element $-1 \in (G/T)(k)$ that acts by $-1$ on $T^*$.  Then we may partition $\Omega \setminus \{ 0 \}$ as $P \coprod -P$ for some set $P$.  \emph{If $|P| > \dim T$, then $-1$ acts nontrivially on $\ker \psi$.}  Indeed, there are $n_\pi \in \Z$ for $\pi \in P$, not all zero, so that $\sum n_\pi \pi = 0$ in $T^*$, which provides an element of $\ker \psi$ that is moved by $-1$.
\end{eg}

\subsection*{The group $\AGL_1$}
The following result will be used for groups of type $C$.

Let $k$ be an algebraically closed field of characteristic $p \ge 0$.   Let $X$ be the variety
of monic polynomials of degree $n$ over $k$.   Of course, $X$ is isomorphic to affine space $k^n$
and can also be identified with $k^n/S_n$ (where the coordinates are just the roots
of the polynomial).   Let $X_0$ be the subvariety of $X$ such
that the coefficient of $x^{n-1}$ is $0$  (i.e., the sum of the roots of $f$ is $0$).     Let 
$G = \AGL_1$, the group with $k$-points $\{ \stbtmat{c}{b}{0}{1} \mid c \in k^\times, b \in k \}$, 
so $G$ is a semidirect product $\Gm \ltimes \Ga$ and is isomorphic to a Borel subgroup of $\PGL_2$.    An element $g \in G$ acts on $k$ by
$y \mapsto cy + b$ and we can extend this to an action on $X$ (by acting on each
root of $f$).   Note that $G$ preserves $X_0$ if and only if $p$ divides $n$.  In any
case $\Gm$ does act on $X_0$. 

\begin{lem}  \label{lem:polys}  If $p$ does not divide $n > 2$, then 
$\Gm$ acts generically freely on 
$X_0$.    If $p$ divides $n$ and $n > 4$,  then $G$ acts generically freely on $X_0$.
\end{lem}

\begin{proof}  We just give the proof of the group of $k$-points.  The proof for the Lie
algebra is identical.  

Note that if $c = (0,c) \in k^\times$ and $f$ has distinct roots, then $c$ fixes $f$ implies that
$c^{n(n-1)} =1 $ since $c$ preserves the discriminant of $f$.   In particular, there
are only finitely many possibilities for $c$.  

Note that the dimension of the fixed point space of multiplication by $c$ on $X$
has dimension at most $n/2$ and so the fixed point space on $X_0$ is a proper
subvariety (because $n > 2$) and has codimension at least $2$ if $n > 4$.

If $p$ does not divide $n$, then we see that there are only finitely many elements
of $k^\times$ which have a fixed space which intersects the open subvariety of $X_0$
consisting of elements with nonzero discriminant.  Thus, the finite union of these fixed spaces
is contained in a proper subvariety of $X_0$ whence for a generic point the
stabilizer is trivial. 

Now suppose that $p$ divides $n$. 
Then translation by $b$ has a fixed space of dimension
$n/p \le n/2$ on $X$ and so similarly the fixed space on $X_0$ has codimension 
at least $2$ for $n > 4$.   

There is precisely one conjugacy class of nontrivial unipotent
elements in $G$ and this class has dimension $1$. Thus the union of all fixed spaces of nontrivial unipotent elements of 
$G$ is contained in a hypersurface for $n > 4$.    Any semisimple element of $G$ is conjugate
to an element of $k^\times$ (i.e., to an element of the form $(0,c)$) and so there are only finitely many such conjugacy classes which 
have fixed points on the locus of polynomials with nonzero discriminant.
Again,   since each class is $1$-dimensional and each fixed space has codimension greater than $1$,
 we see that the union of all fixed spaces is contained in a hypersurface of $X_0$ for $n > 4$.
  \end{proof}
  
 If $n=4$ and $p=2$, then any $f \in X_0$ is fixed by a translation and so the action is
  not generically free. 

\section{Essential dimension}

The essential dimension of an affine group scheme $G$ over a field $k$ can be defined as follows.  For each extension $K$ of $k$, write $H^1(K, G)$ for the cohomology set relative to the fppf (= faithfully flat and finitely presented) site.  For each $x \in H^1(K, G)$, we define $\ed(x)$ to be the minimum transcendence degree of $K_0$ over $k$ for $k \supseteq K_0 \supseteq K$ such that $x$ is in the image of $H^1(K_0, G) \to H^1(K, G)$.  The \emph{essential dimension} of $G$, denoted $\ed(G)$, is defined to be $\max \ed(x)$ as $x$ varies over all extensions $K$ of $k$ and all $x \in H^1(K, G)$.

If $V$ is a representation of $G$ on which $G$ acts generically freely, then $\ed(G) \le \dim V - \dim G$, see, e.g., \cite[Prop.~3.13]{M:ed}.  We can decrease this bound somewhat by employing the following.

\begin{lem} \label{ed.proj}
Suppose $V$ is a representation of an algebraic group $G$.  If there is a $G$-equivariant dominant rational map $V \dashrightarrow X$ for a $G$-variety $X$ on which $G$ acts generically freely, then $\ed(G) \le \dim X - \dim G$.
\end{lem}

\begin{proof}
Certainly, $G$ must act generically freely on $V$.
In the language of \cite{DuncanRei} or \cite[p.~424]{M:ed}, then, $V$ is a versal and generically free $G$-variety and the natural map $V \dashrightarrow X$ is a $G$-compression.  Therefore, referring to \cite[Prop.~3.11]{M:ed}, we find that $\ed(G) \le \dim X - \dim G$.
\end{proof}

The following lemma was pointed out in \cite{BRV:arxiv}.

\begin{lem} \label{BRV.lem}
Let $V$ be a faithful representation of a (connected) reductive group $G$.  Then $\ed(G) \le \dim V$.
\end{lem}

\begin{proof}
Fix a maximal torus $T$ in $G$.  By \cite{ChGiRei:semi} there is a finite sub-$k$-group-scheme $S$ of $N_G(T)$ so that the natural map of fppf cohomology sets $H^1(K, S) \to H^1(K, G)$ 
is surjective for every extension $K$ of $k$, so $\ed(G) \le \ed(S)$.

For generic $v$ in $V$, the stabilizer $S_v$ of $v$ in $S$ has $\Lie(S_v) \subseteq \Lie(S)_v \subseteq \Lie(T)_v = 0$, so $S_v$ is smooth.  As the finite group $S(k)$ acts faithfully on $V$,  $S(k)$ acts generically freely, so $\ed(S) \le \dim V$.
\end{proof}

While the lemma gives cheap upper bounds on $\ed(G)$, it is not sufficient to deduce Theorem \ref{big.O} even in the coarse sense of big-$O$ notation: the minimal faithful representations of $\SL_n / \mu_m$ are too big for $3 \le m < n$, being at least cubic in $n$ whereas the dimension of the group is $n^2 - 1$.

\section{The short root representation} \label{short.sec}

Let $G$ be an adjoint simple algebraic group and put $V$ for the Weyl module with highest weight the highest short root.  Fixing a maximal torus $T$ in $G$, the weights of this representation are 0 (with some multiplicity) and the short roots $\Omega$ (each with multiplicity 1), and we put $\Vb$ for $V$ modulo the zero weight space.  It is a module for $N_G(T)$.

\begin{prop} \label{short}
Suppose $k$ is algebraically closed.
If $G$ is of type $A_n$ ($n \ge 2$), $C_n$ ($n \ge 3)$, $D_n$ ($n \ge 4$), $E_6$, $E_7$, $E_8$, or $F_4$, then $N_G(T)$ acts generically freely on $\P(\Vb)$ and
\[
\ed(N_G(T)) \le |\Omega| - \dim T - 1.
\]
\end{prop}

The inequality in the proposition is reminiscent of the one in \cite[Th.~1.3]{Lemire}.

\begin{eg} \label{A1.eg}
The group $\PGL_n$ is adjoint of type $A_{n-1}$ and we identify it with the quotient of $\GL_n$ by the invertible scalar matrices.  We may choose $T \subset \PGL_n$ to be the image of the diagonal matrices and $N_G(T)$ is the image of the monomial matrices.  The representation $\Vb$ is the space of matrices with zeros on the diagonal, on which $N_G(T)$ acts by conjugation.

With this notation, for type $A_1$, the stabilizer in $N_G(T)$ of a generic element $v := \stbtmat{0}{x}{y}{0}$ of $\Vb$ is $\Z/2$, with nontrivial element the image of $v$ itself.  So type $A_1$ is a genuine exclusion from the proposition.
\end{eg}

\begin{proof}[Proof of Proposition \ref{short}]
It suffices to prove that $N_G(T)$ acts generically freely, for then the inequality follows by Lemma \ref{ed.proj}.  We apply Lemma \ref{proj}.  

For every short root $\alpha$, there is a short root $\beta$ such that $\qform{\beta, \alpha} = \pm 1$.  
(If $\alpha$ is simple, take $\beta$ to be simple and adjacent to $\alpha$ in the Dynkin diagram.  Otherwise, $\alpha$ is in the Weyl orbit of a simple root.)  Thus, the kernel of $T \to \PGL(\Vb)$ is contained in the kernel of $\beta - s_\alpha(\beta) = \pm\alpha$.  As the lattice generated by the short roots $\alpha$ is the root lattice $T^*$, it follows that $T$ acts generically freely on $\P(\Vb)$.

So it suffices to verify \ref{proj}\eqref{proj.ker}.  Fix $w \ne 1$ in the Weyl group; we find short roots $\chi_1, \ldots, \chi_r$ such that $\sum \chi_i = 0$ and the set $\{ \chi_i \}$ is not $w$-invariant.

If $w = -1$, then take $\chi_1$, $\chi_2$ to be non-orthogonal short simple roots.  They generate an $A_2$ subsystem and we set $\chi_3 := -\chi_1 - \chi_2$.  (Alternatively, apply Example \ref{proj.m1}.)  This proves the claim for type $F_4$: the kernel of the $G/T$-action on $\ker \psi$ is a normal subgroup of the Weyl group not containing $-1$, and therefore it is trivial.  

If $G$ has type $A$, $D$, or $E$, then all roots are short.  As $w \ne \pm 1$, there is a short simple root $\chi_1$ such that $w(\chi_1) \ne \pm \chi_1$.  (Indeed, otherwise there would be simple roots $\alpha, \alpha'$ such that $w(\alpha) = \alpha$, $w(\alpha') = -\alpha'$, yet $\alpha$ and $\alpha'$ are adjacent in the Dynkin diagram.)  Take $\chi_2 = -\chi_1$.

For type $C_n$ ($n \ge 3$), as in \cite{Bou:g4} we may view the root lattice $\Z[\Phi]$ as contained in a copy of $\Z^n$ with basis $\e_1, \ldots, \e_n$.  If the kernel of the $G/T$-action on $\ker \psi$ does not contain $-1$, it contains the group $H$ isomorphic to $(\Z/2)^{n-1}$ consisting of those elements that send $\e_i \mapsto -\e_i$ for an even number of indexes $i$ and fix the others.  Taking $\chi_1 = \e_1 - \e_2$, $\chi_2 = \e_2 - \e_3$, and $\chi_3 = -\chi_1 -\chi_2$ gives a set $\{ \chi_i \}$ not stabilized by $H$.
\end{proof}

\section{Groups of type $C$: Proof of Theorem \ref{PSp}}

Let $G$ be the adjoint group of type $C_n$ for $n > 3$ over an algebraically closed field $k$ of characteristic 
$p$.    Let $W:=L(\omega_2)$ be the irreducible module for $G$ with highest weight $\omega_2$ where
$\omega_2$ is the the second fundamental dominant weight (as numbered in \cite{Bou:g4}).     We view $W$ as the unique irreducible nontrivial $G$-composition factor of
$Y:=\wedge^2(V)$ where $V$ is the natural module for $\Sp_{2n}$.   We recall that $Y = W \oplus k$ if $p$ does not divide $n$.
If $p$ divides $n$, then $Y$ is uniserial of length $3$ with 1-dimensional socle and radical.  Any element in $Y$ has characteristic polynomial $f^2$ where 
$f$ has degree $n$, and the radical $Y_0$ of $Y$ is the set of elements with the roots of $f$ summing to $0$.  (Note that aside from characteristic $2$,  $Y_0$
are the elements of trace $0$ in $Y$.)   

 In particular, 
$\dim W = 2n^2-n - 1$ if $p$ does not divide $n$ and $\dim W = 2n^2-n-2$ if $p$ does divide $n$.

As in \cite{GoGu}, we view  $Y$ as the set of skew adjoint operators on $V$ with respect to the alternating form defining $\Sp_{2n}$ with
$G$ acting as conjugation on $Y$. 

\begin{prop}   \label{Sp}  If $n > 3$ and $(n,p) \ne (4,2)$, then 
$G$ acts generically freely on $W \oplus W$ and on 
$\P(W) \times \P(W)$.  
\end{prop}

\begin{proof}    
Any element $y \in Y$ is conjugate to an element of the form $\diag(A, A^{\top})$ acting on a direct sum of totally singular subspaces.
A generic element of $Y$ is thus an element where $A$ is semisimple regular.   Writing $V = V_1 \perp \ldots \perp V_n$ where the $V_i$
are 2-dimensional nonsingular spaces on which $y$ acts as a scalar, we see that a generic point of $Y$ has stabilizer 
(as a group scheme) $\Sp_2^{\times n} = \SL_2^{\times n}$ in $\Sp_{2n}$
(and by \cite{GoGu}, this precisely the intersection of two generic conjugates of $\Sp_{2n}$ in
$\SL_{2n}$).    The same argument shows that this is true for a generic point of $Y_0$.

In particular if $p$ does not divide $n$, the same is true for $W=Y_0$.  
It follows by Lemma \ref{lem:polys} that for generic $w \in W$,  $gw =cw$
for $g \in G$ and $c \in k^x$ implies that $c=1$.  Thus, the stabilizer of a generic
point in $\P(W)$ still has stabilizer $\Sp_2^{\times n} = \SL_2^{\times n}$.

If $p$ does divide $n$, then $W = Y_0/k$ where we identify $k$ with the scalar matrices in $Y$.
We claim that (for $n > 4$) the generic stabilizer is still 
$\Sp_2^{\times n} = \SL_2^{\times n}$ on $\P(W)$.   
Again, this follows by Lemma \ref{lem:polys} since if $gw = cw + b$ with
$b,c \in k$ and $g \in G$, then for $w$ generic,  $b=0$ and $c=1$.

It is straightforward to see that the same is true for $W$, because for a generic point anything stabilizing $y$ modulo scalars
must stabilize $y$ (see Lemma \ref{lem:polys}).    Thus, in all cases, the generic
stabilizer of a point in $\P(W) \times \P(W)$ is the same as for $Y \oplus Y$.  

Consider $\GL_{2n}$ acting on $Y \oplus Y \oplus Y$.   The stabilizer of a generic point of $Y$ is clearly a conjugate of $\Sp_{2n}$.
It follows from \cite{GoGu} that the stabilizer of a generic element of $Y \oplus Y \oplus Y$ is central. (The result is only stated
for the algebraic group but precisely the same proof holds for the group scheme.)     Thus, the same holds for $\Sp_{2n}$ acting
on $Y \oplus Y$ and so also on $\P(W) \times \P(W)$.  The result follows. 
\end{proof}

We can now improve and extend Lemire's bound for $\ed(\PSp_{2n})$ from \cite[Cor.~1.4]{Lemire} both numerically and to fields of all characteristics.

\begin{proof}[Proof of Theorem \ref{PSp}]   
For the first inequality, the group $\GL_{2n}/\mu_2$ has an open orbit on $\wedge^2(k^{2n})$, and the stabilizer of a generic element is $\PSp_{2n}$.  Consequently, the induced map $H^1(K, \PSp_{2n}) \to H^1(K, \GL_{2n}/\mu_2)$ is surjective for every field $K$, see \cite[Th.~9.3]{G:lens} or \cite[\S{III.2.1}]{SeCG}.  (Alternatively, the domain classifies pairs $(A, \sigma)$ where $A$ is a central simple algebra of degree $2n$ and exponent 2 and $\sigma$ is a symplectic involution \cite[29.22]{KMRT}, and the codomain classifies central simple algebras of degree $2n$ and exponent $2$.  The map is the forgetful one $(A, \sigma) \mapsto A$.)  Thus $\ed(\GL_{2n}/\mu_2) \le \ed(\PSp_{2n})$.

For the second inequality, 
assume that $n \ge 4$ and if $n=4$, then $p \ne 2$. 
As $\PSp_n$ acts generically freely on $\P(W) \times \P(W)$,
$\ed(G) \le  2(\dim \P(W)) - \dim G$ by Lemma \ref{ed.proj}.  Theorem \ref{PSp} follows, because 
 $\dim \P(W) = 2n^2 - n - \delta$ where $\delta = 3$ if $p$ divides $n$ and 2 otherwise.
 
If $n=4$ and $p=2$, $G$ still acts generically freely on $Y_0 \oplus Y_0$.  
Indeed, arguing as above we see that $G$ acts generically freely on 
$\P(Y_0) \times \P(Y_0)$  and the result follows in this case. 
\end{proof} 

\section{Groups of type $A$: proof of Theorem \ref{A}} \label{A.sec}

For the proofs of Theorems \ref{big.O}, \ref{A}, and \ref{ed.thm}, we use the fact that $\ed(N_G(T)) \ge \ed(G)$, because for every field $K \supseteq k$, the natural map $H^1(K, N_G(T)) \to H^1(K, G)$ is surjective (which in turn holds because, for $K$ separably closed, all maximal $K$-tori in $G$ are $G(K)$-conjugate).

Let $T$ be a maximal torus in $G := \PGL_n$ for some $n \ge 4$.  The representation $\Vb$ of $N_G(T)$ from \S\ref{short.sec} may be identified with the space of $n$-by-$n$ matrices with zeros on the diagonal.  It decomposes as $\Vb = \oplus_{i=1}^n W_i$, where $W_i$ is the subspace of matrices whose nonzero entries all lie in the $i$-th row; $N_G(T)$ permutes the $W_i$'s.

\begin{lem} \label{projs}
If $n \ge 4$, then
$N_G(T)$ acts generically freely on $X := \P(W_1) \times \P(W_2) \times \cdots \times \P(W_n)$.
\end{lem}

\begin{proof}
Each element of the maximal torus $T$ is the image of a diagonal matrix $t := \diag(t_1, \ldots, t_n)$ under the surjection $\GL_n \to \PGL_n$.  The kernel of the action of $T$ on $\P(W_i)$ are the elements such that $t_i t_j^{-1}$ are equal for all $j \ne i$.  Thus the kernel of the action on $X$ is the subgroup of elements with $t_i = t_j$ for all $i, j$, so $T$ acts faithfully on $X$.  For generic $x \in X$, 
the identity component of $N_G(T)_x$ is contained in $T_x$, so $\Lie(N_G(T)_x) \subseteq \Lie(T)_x = 0$, i.e., $N_G(T)_x$ is finite \'etale.

To show that the (concrete) group $S$ of $k$-points of $N_G(T)_x$ is trivial, it suffices to check $1 \ne s \in S$ that 
\begin{equation} \label{projs.1}
\dim s^T + \dim X^s < \dim X
\end{equation}
(compare, for example, \cite[10.2, 10.5]{GG:simple}). As $s \ne 1$, it permutes the $W_i$'s nontrivially.  If $s$ moves more than two of the $W_i$'s, then 
\[
\dim X - \dim X^s \ge 2 \dim \P(W_i) = 2(n-2).
\]
But of course $\dim s^T \le n - 1$, verifying \eqref{projs.1} for $n \ge 4$.

If $s$ interchanges only two of the $W_i$'s, i.e., it is a transposition, then $\dim X - \dim X^s = n-2$, but $\dim s^T = 1 < n - 2$, and again \eqref{projs.1} has been verified.
\end{proof}

\begin{eg} \label{A1.ed}
$\ed(\PGL_2) = 2$, regardless of $\car k$, so type $A_1$ is a genuine exception to Theorems \ref{big.O} and \ref{A} (as $\dim \PGL_2 - 2(\rk \PGL_2) - 1 = 0$).
Indeed, $H^1(k, G)$ classifies quaternion algebras over $k$, i.e., the subgroup $\Z/2 \times \mu_2$ of $\PGL_2$ gives a surjection in flat cohomology $H^1(k, \Z/2) \times H^1(k, \mu_2) \to H^1(k, \PGL_2)$, so $\ed(\PGL_2) \le 2$.  On the other hand, the connecting homomorphism $H^1(K, \PGL_2) \to H^2(K, \mu_2)$, which sends a quaternion algebra to its class in the 2-torsion of the Brauer group of $K$, is nonzero for some extension $K$, and therefore also $\ed(\PGL_2) \ge 2$.

Entirely parallel comments apply to $\PGL_3$, in which case the surjectivity $H^1(k, \Z/3) \times H^1(k, \mu_3) \to H^1(k, \PGL_3)$ is due to Wedderburn \cite[19.2]{KMRT}.  Thus $\ed(\PGL_3) = 2$ and $\PGL_3$ is a genuine exception to Theorem \ref{A}.
\end{eg}

The proof of Theorem \ref{A} requires a couple more lemmas.

\begin{lem} \label{GLSL}
Suppose $1 \to A \to B \to C \to 1$ is an exact sequence of group schemes over $k$.  If $H^1(K, C) = 0$ for every $K \supseteq k$, then $\ed(B) \le \ed(A) \le \ed(B) + \dim C$.
\end{lem}

\begin{proof}
For every $K$, the sequence
\begin{equation} \label{GLSL.1}
C(K) \to H^1(K, A) \to H^1(K, B) \to 1
\end{equation}
is exact.  From here the argument is standard.
The surjectivity of the middle arrow gives the first inequality.  For the second, take $\alpha \in H^1(K, A)$.  There is a field $K_0$ lying between $k$ and $K$ such that $\trdeg_k K_0 \le \ed(B)$ and an element $\alpha_0 \in H^1(K_0, A)$ whose image in $H^1(K, B)$ agrees with that of $\alpha$.  Thus, there is a $c \in C(K)$ such that $c \cdot \res_{K/K_0}(\alpha_0) = \alpha$.  There is a field $K_1$ lying between $k$ and $K$ such that $\trdeg_k K_1 \le \dim C$ such that $c$ belongs to $C(K_1) \subseteq C(K)$.  In summary, 
\[
\ed(\alpha) \le \trdeg_k (K_1 K_0) \le \trdeg_k K_1 + \trdeg_k K_0 \le \ed(B) + \dim C.
\]
As this holds for every $K$ and every $\alpha \in H^1(K, A)$, the conclusion follows.
\end{proof}

Lemma \ref{GLSL} applies, for example, when $B$ is an extension of a group $A$ by a quasi-trivial torus $C$, such as when $A = \SL_n / \mu_m$ and $B = \GL_n / \mu_m$.  In that case, one can tease out whether $\ed(\SL_n / \mu_m) = \ed(\GL_n/\mu_m)$ or $\ed(\GL_n/\mu_m) + 1$ by arguing as in \cite{ChM:edpA}.

\begin{lem} \label{SL.xfer}
Suppose $m$ divides $n \ge 2$.  Then
\[
\ed(\GL_n / \mu_m) \le \ed(\PGL_n) + n/m - 1.
\]
\end{lem}

We omit the proof, which is the same as that for \cite[Lemma 7.1]{BaekM} apart from cosmetic details.

\begin{proof}[Proof of Theorem \ref{A}]
In view of Lemmas \ref{projs} and \ref{ed.proj}, we find that 
\[
\ed(\PGL_n) \le \ed(N_G(T)) \le \dim X - \dim N_G(T) = n^2 - 3n + 1.
\]
Therefore Lemma \ref{SL.xfer} gives
\begin{equation} \label{GLSL.bound}
\ed(\GL_n / \mu_m) \le n^2 - 3n + n/m
\end{equation}
and Lemma \ref{GLSL} gives the required bound on $\ed(\SL_n / \mu_m)$.
\end{proof}

\begin{lem} \label{coprime}
Suppose $m$ divides $n$, and write $n = n'q$ where $n'$ and $m$ have the same prime factors and $\gcd(n',q) = 1$.  Then $H^1(K, \GL_n / \mu_m) = H^1(K, \GL_{n'} / \mu_m)$ for every extension $K$ of $k$ and 
$\ed(\GL_n / \mu_m) = \ed(\GL_{n'}/\mu_m)$.
\end{lem}

\begin{proof}
The set $H^1(K, \GL_n / \mu_m)$ is in bijection with the isomorphism classes of central simple $K$-algebras $A$ of degree $n$ and exponent dividing $m$.  As $n'$ and $q$ are coprime, every such algebra can be written uniquely as $A' \ot B$ where $A'$ has degree $n'$ and $B$ has degree $q$ \cite[4.5.16]{GilleSz}.  However, $B$ is split as its exponent must divide $\gcd(q, \exp A)$, i.e., $A \cong M_q(A')$.  That is, $H^1(K, \GL_n / \mu_m) = H^1(K, \GL_{n'}/\mu_m)$.  As this holds for every extension $K$ of $k$, the claim on essential dimension follows.
\end{proof}

\begin{rmk}
One can eliminate $m$ from the bound appearing in Theorem \ref{A} to obtain
\[
\ed(\SL_n / \mu_m) \le n^2 - 3n + 1 + n/4 \quad \text{for $m$ dividing $n \ge 4$.}
\]
To check this, assume $m < 4$.  If $m = 1$, $\ed(\SL_n) = 0$. If $m = 2$, then Theorem \ref{PSp} gives a stronger bound.  

If $m = 3$, then write $n = n'q$ for $n' = 3^a$ for some $a \ge 1$ as in Lemma \ref{coprime}.  If $a = 1$, then $n \ge 6$ and $\ed(\GL_n / \mu_3) = \ed(\PGL_3) = 2$ by Lemma \ref{coprime}, which is less than $n^2 - 3n + n/4$.  If $a > 1$, then $\ed(\GL_n/\mu_3) \le \ed(\PGL_{n'}) + n'/3 - 1$; as $n'$ is odd and $\ge 9$, \cite{LRRS} gives $\ed(\PGL_{n'}) \le \frac12 (n'-1)(n'-2)$, whence the claim.
\end{rmk}

\begin{rmk}
Here is another way to obtain an upper bound on $\ed(\SL_n / \mu_m)$; it is amusing because it requires $\car k = p$ to be nonzero.  Fix an integer $e \ge 1$ and $\e = \pm 1$, and set $m := \gcd(p^e + \e, n)$.  We will show that
\begin{equation} \label{Frob.bd}
\ed(\SL_n / \mu_m) \le n^2 - n + 1.
\end{equation}
To see this, consider the $\GL_n$-module $V := W \ot W^{[e]}$ or $W^* \ot W^{[e]}$, where $W$ is the natural module $k^n$, $[e]$ denotes the $e$-th Frobenius twist, and where we take the first option if $\e = +1$ and the second option if $\e = -1$.  A scalar matrix $x \in \GL_n$ acts on $V$ as $x^{p^e + \e}$, and therefore the action of $\SL_n$ on $V$ gives a faithful representation of $G := \SL_n / \mu_m$.  We consider the action of $N_G(T)$ on $V$ for $T$ a maximal torus in $G$, and apply Lemma \ref{proj} to see that $N_G(T)$ acts generically freely on $V$ and so obtain \eqref{Frob.bd}.
\end{rmk}

\section{Minuscule representations of $E_6^\sc$ and $E_7^\sc$: proof of Theorem \ref{ed.thm}}

Recall that $E_6^\sc$ and $E_7^\sc$ have minuscule representations, i.e., representations where all weights are nonzero and occur with multiplicity 1 and make up a single orbit $\Omega$ under the Weyl group.  For $E_6$ there are two inequivalent choices, both of dimension 27, and for $E_7$ there is a unique one of dimension 56.

\begin{prop} \label{minu}
Let $T$ be a maximal torus in a simply connected group $G$ of type $E_6^\sc$ or $E_7^\sc$ over an algebraically closed field $k$.  Then $N_G(T)$ acts generically freely on $V$ for every minuscule representation $V$ of $G$.
\end{prop}

\begin{proof}
We apply Lemma \ref{proj}.  The map $G \to \GL(V)$ is injective, so $T$ acts faithfully on $V$.  It suffices to verify \ref{proj}\eqref{proj.ker}.

One can list explicitly the weights $\Omega$ of $V$ and find $X = \{ \chi_1, \ldots, \chi_6 \} \subset \Omega$ with $\sum \chi_i = 0$ and $\chi_i \ne \pm \chi_j$ for $i \ne j$.  It suffices, therefore, to check for every minimal normal subgroup $H$ of the Weyl group not containing $-1$, that $HX \ne X$.  For this, it is enough to observe that $H$ has no fixed lines on the vector space $\C[\Phi]$ generated by the roots $\Phi$ (because $\Z[\Omega] = \Z[\Phi]$, so $H$ fixes no element of $\Omega$) and that $H$ has no orbits of size $2, 3, \ldots, 6$ (because its maximal subgroups have index greater than 6).  

For $E_6$, $H$ has order 25920 with largest maximal subgroups of index 27.  For $E_7$, $H$ is isomorphic to $\Sp_6(\F_2)$ with largest maximal subgroups of index 28.  The description of these Weyl groups from \cite[Ch.~IV, \S4, Exercises 2 and 3]{Bou:g4} make it obvious that $H$ does not preserve any line in $\C[\Phi]$.
\end{proof}

\begin{proof}[Proof of Theorem \ref{ed.thm}]
The group $F_4$ has 24 short roots, so by Proposition \ref{short}, we have
\[
\ed(F_4) \le \ed(N_G(T)) \le 24 - 4 - 1 = 19.
\]
For $E_7^\sc$, we apply instead Proposition \ref{minu} to obtain the desired upper bound.

The group $E_6^\sc$ has a subgroup $F_4 \times \mu_3$ such that the 
map in 
cohomology $H^1(K, F_4 \times \mu_3) \to H^1(K, E_6^\sc)$ is surjective for every extension $K \supseteq k$, see \cite[9.12]{G:lens}, hence $\ed(E_6^\sc) \le \ed(F_4) + 1$.
\end{proof}

\section{Proof of Theorem \ref{big.O}} \label{A.sec}

\begin{proof}[Proof of Theorem \ref{big.O}]
Suppose first that $G$ has type $A_{n-1}$, i.e., $G \cong \SL_n / \mu_m$.  Assume $m > 1$ for otherwise $\ed(G) = 0$.  It is claimed that $\ed(G) \le n^2 - 2n$.  As $\ed(\PGL_3) = 2$, we may assume $n \ge 4$.  Combining Theorem \ref{A} with the fact that $1 + n/m \le n$ gives the claim.

Now suppose that $G$ is adjoint.  If $G$ is one of the types covered by Proposition \ref{short}, then  we are done by combining that proposition with the inequality $\ed(G) \le \ed(N_G(T))$. Type $B$ was already addressed in the introduction.  For type $G_2$, the essential dimension is 3 because $H^1(K, G_2)$ is in bijection with the set of 3-Pfister quadratic forms over $K$ for every field $K$ containing $k$ \cite[26.19]{KMRT}.

Now suppose that $G$ is neither type $A$ nor adjoint.  If $G$ has type $B$, then $G$ is a spin group, so there is nothing to prove.  If $G$ has type $C$,  then $G=\Sp_{2n}$ and $\ed(G)=0$.  If $G$ has type $D$, then the only remaining case to consider is $G=\SO_{2n}$ for $n \ge 4$ and then 
$\ed(G) \le 2n - 1< 2n^2 - 3n - 1 = \dim G - 2(\rk G) - 1$.  The two remaining cases are the simply connected groups of type $E_6$ and $E_7$ for which we refer to Theorem  \ref{ed.thm}.
\end{proof}

\section{Generic stabilizer for the adjoint action} \label{adj.sec}

As a complement to the above results, we now calculate the stabilizer in a simple algebraic group $G$ of a generic element in $\Lie(\Ad(G))$.   (Note that, in case $G = \SL_2$, we are discussing the action on $\Lie(\PGL_2)$, not on $\Lie(\SL_2)$, and the two Lie algebras are distinct if $\car k = 2$.)
We include this calculation here because the methods are similar to the previous results.  The results are complementary, in the sense that previously we considered $N_G(T)$ acting on representations with no zero weights, and in this section we consider $N_G(T)$ acting on $\Lie(\Ad(T))$, for which zero is the only weight.  The main result, Proposition \ref{adjoint}, is used in \cite{GG:spin}.

After a preliminary result, we will calculate the stabilizer of a generic element of the adjoint representation.
Let $\Phi$ be an irreducible root system and put $W$ for its Weyl group and $Q$ for its root lattice.  
For each prime $p$, tensoring $Q$ with the finite field $\F_p$ gives a homomorphism
\[
\rho_p \!: \qform{W, -1} \to \GL_{\rk Q}(\F_p).
\]

\begin{lem}  \label{weyl}
The kernel of $\rho_p$ is $(\Z/2)^n$ if $\Phi$ has type $B_n$ for some $n \ge 2$ and $p = 2$.  Otherwise, 
$\ker \rho_p = \qform{-1}$ if $p = 2$ and $\ker \rho_p = 1$ for $p \ne 2$. 
\end{lem}

\begin{proof}   If $p \ne 2$, $\ker \rho_p =1$ by an old theorem of Minkowski (see \cite{Mi87} and also \cite[Lemma 1.1]{Se07}).   So
we may assume that $p=2$.  it also follows by a similar argument that $\ker \rho_2$ is a $2$-group \cite[Lemma 1.1']{Se07}.
 Clearly $-1 \in \ker \rho_2$.
Thus, the result follows immediately for $G$ of type $A_n$ for $n \ne 3$,  $G_2$, or $E_n$, since the only normal $2$-subgroups in
these cases are the subgroup of order $2$ containing $-1$.     It is straightforward to check the result for the groups $A_3=D_3$
and $C_3$.   Note that the root lattice of $D_{n-1}$  is a direct summand of $D_n, n > 3$ and any normal $2$-subgroup of the Weyl group
of $D_n$ of order greater than $2$ intersects the Weyl group of of $D_{n-1}$ in a subgroup of order greater than $2$.
Thus, the result for $D_3$ implies the result for all $D_n$.    Similarly, the result for $C_3$ implies the result for $C_n, n > 3$.  

Finally, suppose $\Phi$ has type $B_n$ for some $n \ge 2$ and $p = 2$.  Viewing $\Z^n$ as having basis $\e_i$ for $1 \le i \le n$, we can embed $\Phi$ in $\Z^n$ by setting the simple roots to be $\alpha_i = \e_i - \e_{i+1}$ for $1 \le i < n$ and $\alpha_n = \e_n$ as in \cite{Bou:g4}.  The Weyl group $W$ is isomorphic to $(\Z/2)^n \rtimes S_n$, where $(\Z/2)^n$ consists of all possible sign flips of the $\e_i$ and $S_n$ acts by permuting the $\e_i$.  The subgroup $(\Z/2)^n$ obviously acts trivially on $Q \ot \F_2$ (since
there is a basis of eigenvectors for $Q$ for this subgroup of exponent $2$).  In fact, $(\Z/2)^n$ is precisely the kernel of the action of $W$ on $Q \ot \F_2$, 
as is easy to check for $n < 5$ and is clear for $n \ge 5$.
\end{proof}

\begin{prop} \label{adjoint}
Let $G$ be a simple algebraic group.  The action of $G$ on $\Lie(\Ad(G))$ has stabilizer in general position $S$, with identity component $S^\circ$ a maximal torus in $G$.  Moreover, $S = S^\circ$ unless $\car k = 2$ and:
\begin{enumerate}
\item $G$ has type $B_n$ for $n \ge 2$; in this case $S/S^\circ \cong (\Z/2)^n$.
\item $G$ has type $A_1$, $C_n$ for $n \ge 3$, $D_n$ for $n \ge 4$ even, $E_7$, $E_8$, $F_4$, or $G_2$; in this case $S/S^\circ \cong \Z/2$ and the nontrivial element acts on $S^\circ$ by inversion.
\end{enumerate}
\end{prop}

\begin{proof} 
Suppose first that $G = \Ad(G)$ and fix a maximal torus $T$ of $G$.   As $G$ is adjoint, the Lie algebra $\Lie(T)$ is a Cartan subalgebra of $\Lie(G)$, and the natural map $G \times \Lie(T) \to \Lie(G)$ is dominant \cite[XIII.5.1, XIV.3.18]{SGA3.2}.  Therefore, it suffices to verify that the stabilizer $S$ in $G$ of a generic vector $t$ in $\Lie(T)$ is as claimed.  The subgroup of $G$ transporting $t$ in $\Lie(T)$ is the normalizer $N_G(T)$ \cite[XIII.6.1(d)(viii)]{SGA3.2}, hence $S$ is the centralizer of $t$ in $N_G(T)$ and it follows that $S^\circ = T$ and $S/S^\circ$ is isomorphic to the group of elements $w$ of the Weyl group fixing $t$, compare \cite[Lemma 3.7]{St:tor}.  As $G$ is adjoint, the element $t$ is determined by its action on $\Lie(G)$, i.e., by the values of the roots on $t$; in particular $w(t) = t$ if and only if $w$ acts trivially on $Q \ot k$.  Lemma \ref{weyl} completes the proof for $G$ adjoint.

In case $G$ is not adjoint, the representation factors through the central isogeny $G \to \Ad(G)$, and $G_t$ is the inverse image of the generic stabilizer in $\Ad(G)$.
\end{proof}

To summarize the proof, the identity component of $C_G(t)$ is $T$ by \cite{SGA3.2}, so $C_G(t)$ is contained in $N_G(T)$ and is determined by its image in the Weyl group $N_G(T)/T$; this statement is included in \cite{St:tor}.  What is added here is the calculation of the component group $C_G(t)/T$, and in particular that it need not be connected.

 One can also compute the generic stabilizer for the action of $G$ on the projective space $\P(\Lie(G))$ of $\Lie(G)$ by the same argument.
If $p = 2$, since $\PGL_n(\F_2) = \GL_n(\F_2)$ we see that the generic stabilizers for $\P(\Lie(G))$ and $\Lie(G)$ are the same.
If $p$ is odd, an easy argument shows that a generic stabilizer is a maximal torus if $-1$ is not in the Weyl group and is just
a maximal torus extended by $-1$ if $-1$ is in the Weyl group. (Clearly $-1$  does act by $-1$ on $\Lie(T)$, $T$ a maximal torus.) 
In any case, the connected component of the stabilizer of a generic line in $\Lie(G)$ is contained in the normalizer of a maximal torus, as we know from \cite{SGA3.2}.

\subsection*{Action of $G$ on $\Lie(G) \oplus \Lie(G)$}

In case $k = \C$, it is well known that an adjoint simple group $G$ acts generically freely on $\Lie(G) \oplus \Lie(G)$. 
However we have also the following:

\begin{eg} \label{A1A1}
Maintaining the notation of Example \ref{A1.eg}, the Lie algebra $\pgl_2$ of $\PGL_2$ may be identified with the Lie algebra $\gl_2$ of 2-by-2 matrices, modulo the scalar matrices.  Write $T$ for the (image of the) diagonal matrices in $\PGL_2$.  A generic element $v \in \pgl_2$ is the image of some $\stbtmat{x}{y}{z}{w}$.  The normalizer of $[v] \in \P(\pgl_2)$ in $N_G(T)$ is $\Z/2$, with nontrivial element the image $g$ of $\stbtmat{0}{y}{-z}{0}$, which satisfies $gv = -v$.  If $\car k = 2$, the same calculation shows that the normalizer of $v \in \pgl_2$ is $\Z/2$.

The subgroup of $\PGL_2$ mapping a generic element of $\Lie(T)$ into $\Lie(T)$ is $N_G(T)$, as was already used in the proof of Proposition \ref{adjoint}.  Therefore, the stabilizer in $G$ of a generic element of $\P(\pgl_2) \oplus \P(\pgl_2)$ equals the stabilizer in $N_G(T)$ of a generic element of $\P(\pgl_2)$, i.e., $\Z/2$.

Moreover, if $\car k = 2$, \emph{the stabilizer in $\PGL_2$ of a generic element in $\pgl_2 \oplus \pgl_2$ is $\Z/2$.}
\end{eg}

We note that this is the only such example.

\begin{prop} \label{double}
Let $G$ be an adjoint simple group.  Then $G$ acts generically freely on $\P(\Lie(G)) \times\P(\Lie(G))$ unless $G$ has type $A_1$.  If $G$ has type $A_1$ and $\car k \ne 2$, then $G$ acts generically freely on $\Lie(G) \oplus \Lie(G)$.
\end{prop}

\begin{proof}
Pick a maximal torus $T$ in $G$.  
The stabilizer in $G$ of a generic element of $\P(\Lie(G)) \times \P(\Lie(G))$ is
contained in the intersection of two generic conjugates of $N_G(T)$.  If $G$ is not of type $A_1$, then this intersection is 1 as in the proof of \cite[Cor.~10]{BGS}.  If $G$ is of type $A_1$ and $\car k \ne 2$, then we apply the preceding example.
\end{proof}

Note that if $p \ne 2$ and we consider the action of $G$ on $\Lie(G)$, then a generic stabilizer is a maximal torus and it is elementary to see
that two generic conjugates of a maximal torus intersect trivially. 

\section{Groups of type $B$ and $D$} \label{spin.sec}

We have not yet discussed upper bounds for the simply connected groups $\Spin_n$ for $n \ge 7$ of type $B_\ell$ for $\ell \ge 3$ or $D_\ell$ for $\ell \ge 4$.  Also, for $\Spin_n$ with $n$ divisible by 4 and at least 12, there is a quotient $\Spin_n / \mu_2$ that is distinct from $\SO_n$; it is denoted $\HSpin_n$ and is known as a half-spin group.

The group $G = \Spin_n$ with $n > 14$ or $\HSpin_n$ with $n > 16$ act generically freely on a (half) spin representation or the sum of a half spin representation and the vector representation $\Spin_n \to \SO_n$ by \cite{AndreevPopov} and \cite{APopov} if $\car k = 0$ and \cite{GG:spin} for all characteristics.  This gives an upper bound on $\ed(G)$, which is an equality if $\car k \ne 2$, see \cite{BRV} and \cite{GG:spin}.

We now give bounds for $\HSpin_{12}$ and $\HSpin_{16}$.

\begin{lem}
For $T$ a maximal torus in $G := \HSpin_n$ for $n$ divisible by $4$ and $n \ge 12$, the group $N_G(T)$ acts generically freely on the half-spin representation of $G$.
\end{lem}

\begin{proof}
Apply Lemma \ref{proj}.
The representation $V$ is minuscule and $T$ acts faithfully because $G$ does so.  The element $-1$ of the Weyl group acts nontrivially on $\ker \psi$ by Example \ref{proj.m1} because $\frac12 \dim V = 2^{n/2-2} > n/2 = \dim T$.  As $-1$ is contained in every nontrivial normal subgroup of the Weyl group, the proof is complete.
\end{proof}

\begin{cor}
Over every algebraically closed field,
\[
\ed(\HSpin_{12}) \le 26 \quad \text{and} \quad \ed(\HSpin_{16}) \le 120.
\]
\end{cor}

The remaining groups are $\Spin_n$ with $7 \le n \le 14$.  In case $\car k \ne 2$, the precise essential dimension is known by Rost, see \cite{Rost:14.1}, \cite{Rost:14.2}, and \cite{G:lens}.  The same methods, combined with the calculations of the generic stabilizers from \cite{GG:spin}, will provide  upper bounds for $\ed(\Spin_n)$ in case $\car k = 2$.  But these methods require detailed arguments, so for our purposes we note simply that $\Spin_n$ acts faithfully on the spin representation for $n$ odd and on the direct sum of the vector representation and a half-spin representation for $n$ even; Lemma \ref{BRV.lem} then provides an upper bound on $\ed(\Spin_n)$.
This completes the task of giving an upper bound on $\ed(G)$ for every simple algebraic group $G$ over an algebraically closed field $k$.

{\small\subsection*{Acknowledgements} We thank the referee, Zinovy Reichstein, and Mark MacDonald  for their helpful comments, which greatly improved the paper.}

\bibliographystyle{amsalpha}
\bibliography{skip_master}

\end{document}